\documentclass[11pt]{amsproc}

\usepackage{mathtext}
\usepackage[cp1251]{inputenc}

\usepackage{bm}
\usepackage[dvips]{graphicx}
\usepackage{amsmath}
\usepackage{amssymb}
\usepackage{amsxtra}

\usepackage{epsfig}
\usepackage{epic}
\usepackage{eepic}
\usepackage{graphics}
\usepackage{graphicx}
\usepackage{subfigure}

\usepackage{caption}
\def\supp{\operatorname{supp}}
\def\sinc{\operatorname{sinc}}

\def\N{{{\Bbb N}}}
\def\Z{{{\Bbb Z}}}

\def\R{{\Bbb R}}
\def\C{{\Bbb C}}
\def\l{{\lambda }}
\def\a{{\alpha }}
\def\D{{\Delta }}

\def\a{{\alpha}}
\def\b{{\beta}}
\def\d{{\delta}}

\def\s{{\sigma}}
\def\vp{{\varphi}}
\def\t{{\theta }}
\def\g{{\gamma }}
\def\w{{\omega }}

\def\){\right)}
\def\({\left(}
\def\supp{\operatorname{supp}}

\numberwithin{equation}{section}
\newtheorem{theorem}{Theorem}[section]

\newtheorem{corollary}[theorem]{Corollary}
\newtheorem{lemma}[theorem]{Lemma}

\newtheorem{proposition}[theorem]{Proposition}
\newtheorem{remark}[theorem]{Remark}

\par

\begin{document}

\title[]{Special measures of smoothness for \\ approximation by  sampling operators in $L_p(\R^d)$}

\author[Yurii
Kolomoitsev]{Yurii
Kolomoitsev$^{\text a, *}$}
\address{Institute for Numerical and Applied Mathematics, G\"ottingen University, Lotzestr. 16-18, 37083 G\"ottingen, Germany
}
\email{kolomoitsev@math-uni.goettingen.de}

\thanks{$^\text{a}$Institute for Numerical and Applied Mathematics, G\"ottingen University, Lotzestr. 16-18, 37083 G\"ottingen, Germany}

\thanks{$^*$Supported by the German Research Foundation, project KO 5804/3-1}


\thanks{E-mail address: kolomoitsev@math-uni.goettingen.de}

\date{\today}

\begin{abstract}
Traditional measures of smoothness  often fail to provide accurate $L_p$-error estimates for approximation by sampling or interpolation operators, especially for functions with low smoothness. To address this issue, we introduce a modified measure of smoothness that incorporates the local behavior of a function at the sampling points through the use of averaged operators.
With this new tool, we obtain matching direct and inverse error estimates for a wide class of sampling operators and functions in
$L_p$ spaces. Additionally, we derive a criterion for the convergence of sampling operators in $L_p$, identify conditions that ensure  the exact rate of approximation, construct realizations of $K$-functionals based on these operators, and study the smoothness properties of sampling operators. We also demonstrate how our results apply to several well-known operators, including the classical Whittaker-Shannon sampling operator, sampling operators generated by $B$-splines, and those based on the Gaussian.
\end{abstract}

\maketitle

\section{Introduction}

The present paper investigates the approximation properties of general linear sampling operators in the spaces  $L_p(\R^d)$, $1\le p<\infty$. Numerous studies have been devoted to the approximation of functions in these spaces using sampling operators. The works most closely related to the present investigation include, for example,~\cite{BBSV06}, \cite{BMSVV14}, \cite{B88}, \cite{C23}, \cite{H89}, \cite{KL23}, \cite{KP21}, \cite{KS20}, \cite{KS21}, \cite{Sk1}, \cite{SS00}. In most of these studies, it is typically assumed that the function
$f$ to be approximated is sufficiently smooth or, in more general settings, that $f$ is continuous and bounded.
In this context, several types of direct error estimates have been established using different smoothness characteristics such as membership in  Besov and Sobolev spaces, special moduli of smoothness, or the decay of the Fourier transform.
To illustrate this, we present some of these estimates for one of the most classical and widely studied sampling and interpolation operators, given by
\begin{equation*}
  {S}_{\s}f(x):=\sum_{k\in\Z}  f(\s^{-1}k)\, {\rm sinc}(\s x-k), \quad{\rm sinc}(x):=\frac{\sin\pi x}{\pi x}.
\end{equation*}
The following property is known (see, e.g., \cite{SS00} and~\cite{Si92}): Let $\a>1/p$ and $1<p<\infty$.
Then a function $f$ belongs to the Besov space $B_{p,\infty}^\a(\R)$ \footnote{As usual, we write $f\in B_{p,q}^\a(\R)$ if $f\in L_p(\R)$ and $(\int_{1}^\infty (\frac{\w_{[\a]+1}(f,t)_p}{t^{\a}})^q\frac{dt}{t})^{1/q}<\infty$ with the standard modification when $q=\infty$.} if and only if $f$ is continuous and satisfies
$$
\Vert f-{S}_\s f\Vert_{L_p(\R)}= \mathcal{O}(\s^{-\a})\quad\text{as}\quad \s\to\infty.
$$
This raises the question of what can be said about this property in the case $\a\le 1/p$.
It is known, for example, from~\cite{SS00} that if $f$ belongs to the Besov space $B_{p,1}^{1/p}(\R)$ where $1<p<\infty$, then
\begin{equation}\label{S1+}
  \Vert f-{S}_\s f\Vert_{L_p(\R)}\le \frac C{\s^{1/p}}\int_{0}^{1/\s}\frac{\omega_r(f,t)_p}{t^{1/p}}\frac{{\rm d} t}{t},  \quad r=1,2,\dots,
\end{equation}
where $\omega_r(f,t)_p$ is the classical modulus of smoothness of order $r \in \N$, and the constant $C$ is independent of $f$ and $\s$, see also~\cite{KS21}.
More general error estimates can be obtained in terms of the $\tau$-modulus of smoothness. Typical results in this direction, which in particular imply~\eqref{S1+}, were established in~\cite{BBSV06} and \cite{SS00}. It was shown that for every bounded function $f\in L_p(\R)$, $1<p<\infty$, satisfying certain additional natural conditions, the following inequality holds:
\begin{equation}\label{S2}
  \Vert f-{S}_\s f\Vert_{L_p(\R)}\le C\tau_s(f,\s^{-1})_p, \quad s=1,2,\dots,
\end{equation}
where $\tau_s(f,\s^{-1})_p$ is the so-called $\tau$-modulus, also known as the averaged modulus of smoothness, and the constant $C$ is independent of $f$ and $\s$. 
Despite the relatively simple form of the error estimates~\eqref{S1+} and~\eqref{S2} and their applicability to a wide range of functions, these inequalities are generally not sharp in the sense of inverse estimates for functions with low smoothness, specifically when $\a\le 1/p$. See, for example,~\cite{KL23}, \cite{BXZ92}, and~\cite{WS03} for a related discussion in the periodic case.

Now let us consider a continuous analogue of the discrete operator $S_\s$ defined by
$$
M_\s f(x):=\s \int_\R f(y)\sinc(\s(t-y))dy.
$$
The following properties are well-known (see, e.g.,~\cite{SS00}): Let $f\in L_p(\R)$, $1<p<\infty$, and $s\in \N$.
Then there exists a constant $C$, independent of $f$ and $\s$, such that
\begin{equation}\label{intrleJB.1}
\|f-M_\s f\|_{L_p(\R)}\le C\w_s(f,\s^{-1})_p
\end{equation}
and
      \begin{equation}\label{intrleJB.2}
         \w_s(f,\s^{-1})_p\le \frac{C}{\s^s}\sum_{\nu=0}^{[\s]}(\nu+1)^{s-1}\|f-M_\nu f\|_{L_p(\R)}.
      \end{equation}
In particular, we have $f\in B_{p,\infty}^\a(\R)$, $\a>0$, if and only if $f\in L_p(\R)$ and
$$
\Vert f-M_\s f\Vert_{L_p(\R)}= \mathcal{O}(\s^{-\a}).
$$
It follows from the inverse approximation inequality (see, e.g., Lemma~\ref{leJB+} below) that if the operator $M_\nu$ is replaced by $S_\nu$, then the same inequality remains valid for any function $f\in L_p(\R)$ for which the corresponding sampling operator is well defined. At the same time, it is clear that inequality~\eqref{intrleJB.1} does not hold when the sampling operator $S_\s$ is used instead of $M_\s$.
Furthermore, although we have the direct estimate~\eqref{S2} in terms of the $\tau$-modulus of smoothness, the corresponding inverse estimate, analogue to~\eqref{intrleJB.2}, does not hold when both the sampling operator and the $\tau$-modulus are used.
This implies that neither the classical modulus of smoothness nor $\tau$-modulus is suitable for establishing direct and inverse approximation inequalities for sampling operators. As a result, they are not appropriate for characterizing functions that satisfy the estimate $\|f-S_\s f\|_{L_p(\R)}=\mathcal{O}(\s^{-\a})$ when $\a\le 1/p$. See also~\cite{KL23} for a more detailed discussion of this issue in the periodic setting.

The main purpose of this paper is to propose an improved approach to describing function smoothness in the context of approximation by sampling operators, aiming to overcome the known  limitations of conventional measures of smoothness.
To this end, we follow the ideas of~\cite{KL23} and propose a modification of the classical modulus of smoothness $\w_s(f,\d)_p$ by using averaged operators to incorporate information about the local behavior of the function $f$ at sampling points. Using this modification, we obtain matching direct and inverse approximation estimates for general linear sampling operators valid for every function in $L_p$ for which the corresponding operator is well defined. We also derive a criterion for their convergence, identify conditions that ensure the exact rate of convergence in $L_p$, construct realizations of
$K$-functionals based on these operators, and study the smoothness properties of sampling operators.
In particular, as a corollary of the main results of this paper,  we show that for any $f\in L_p(\R)$, $1<p<\infty$, satisfying $\sum_k |f(\s^{-1}k)|^p<\infty$, and for $s,r\in\N$ with $s\le 2r$, and $\a\in (0,s)$, the following properties are equivalent:

\medskip

  \begin{enumerate}
    \item[$(i)$]  $\displaystyle\|f-S_\s^{} f\|_{L_p(\R)}=\mathcal{O}(\s^{-\a})$,
    \item[$(ii)$]  $\displaystyle \(\s^{-1}\sum_{k\in \Z}|f_{1/2\s,\,r}(\s^{-1}k)-f(\s^{-1}k)|^p\)^{1/p}+\w_s(f,\s^{-1})_{L_p(\R)}=\mathcal{O}(\s^{-\a})$,
    \item[$(iii)$] $\displaystyle\|(S_\s^{} f)^{(s)}\|_{L_p(\R)}=\mathcal{O}(\s^{s-\a})$,
  \end{enumerate}

\medskip

\noindent where $f_{\d,r}$ denotes the averaged operator of order $r$, see~\eqref{SB3} for its definition.

This paper is organized as follows. Sections~2 and~3 provide the necessary definitions and auxiliary results needed to formulate the main theorems. In Section~4, we present the main results of the paper along with several important corollaries. Section~5 is devoted to the proofs of these results. In Section~6, we illustrate the theory with applications to several specific operators, including the classical Whittaker–Shannon sampling operator, sampling operators generated by $B$-splines, and those based on the Gaussian function.


\medskip

\textbf{Notation.} By $L_p(\R^d)$, $1\le p<\infty$, we denote the space of all finite-valued, measurable functions $f : \R^d \to \R$ such that
$$
\|f\|_p=\|f\|_{L_p(\R^d)}=\(\int_{\R^d} |f(x)|^pdx\)^{1/p}<\infty.
$$
We emphasize that we do not identify functions that are equal almost everywhere.  That is, each function in $L_p(\R^d)$ is determined by its values at each point in $\R^d$.
As usual, $W_p^r(\R^d)$ with $r\in \N$ and $1\le p<\infty$, denotes the Sobolev space equipped with the norm $\|f\|_{W_p^r}=\|f\|_p+|f|_{W_p^r}$,
where the semi-norm $|f|_{W_p^r}$ is given by
$$
|f|_{W_p^r}=\sum_{|\a|_1=r}\|D^\a f\|_p,\quad D^\a=\frac{\partial^{|\a|_1}}{\partial x_1^{\a_1}\dots\partial x_d^{\a_d}}.
$$
In what follows, $X_\s$ with some  $\s>0$, denotes a finite or infinite set of points in $\R^d$.
We assume that the points in $X_\s$ satisfy the following property: there exists $\g\in (0,1]$, independent of $\s$, such that
\begin{equation*}
  \min_{\xi,\xi'\in X_\s,\, \xi\neq\xi'}|\xi-\xi'|>\frac{2\g}{\s}.
\end{equation*}
For a given function $f$ and the set of points  $X_\s\subset\R^d$, the discrete semi-norm $\|f\|_{\ell_p(X_\s)}$, $1\le p<\infty$, is defined by
$$
\|f\|_{\ell_p(X_\s)}=\(\frac1{\s^d}\sum_{\xi\in X_\s}|f(\xi)|^p\)^{1/p}.
$$
We will write that $f\in \ell_p(X_\s)$ if $\|f\|_{\ell_p(X_\s)}<\infty$.
The class of band-limited functions  $\mathcal{B}_p^\s(\R^d)$, $\s>0$, in the space $L_p(\R^d)$, is given by
%
$$
\mathcal{B}_p^\s(\R^d)=\left\{\varphi  \in L_p(\R^d)\cap L_1(\R^d)\,:\,\supp\;\widehat{\varphi} \subset [-\s,\s]^d\right\},
$$
where
$$
\widehat{\vp}(x) = \int_{\R^d} \vp(y) e^{-2\pi i(x,y)} dy
$$
is the Fourier transform of $\vp$.
The open ball of radius $\d$ centered at a point $x\in \R^d$ is denoted by $B_\d(x)$, i.e.,
$$
B_\d(x)=\{y\in \R^d\,:\, |y-x|<\d\}\quad\text{and}\quad B_\d=B_\d({0}).
$$
Throughout the paper, we use the notation
$
\, A \lesssim B,
$
with $A,B\ge 0$, for the estimate
$\, A \le C\, B,$ where $\, C$ is a positive constant independent of
the essential variables in $\, A$ and $\, B$ (usually, $f$, $\s$, and $\d$).
If $\, A \lesssim B$
and $\, B \lesssim A$ simultaneously, we write $\, A \asymp B$ and say that $\, A$
is equivalent to $\, B$.

\section{Main definitions and preliminary results}

The main object of this paper is the linear sampling operator $G_\s$, defined by
\begin{equation*}
  G_\s f(x)=\sum_{\xi\in X_\s} f\big(\xi)\psi_{\xi,\s}(x),
\end{equation*}
where 
$\psi_{\xi,\s}$ are appropriate kernel functions (e.g., band-limited functions, splines, Gaussian, etc.).
Let $1\le p<\infty$ be given. In what follows, we impose the following general assumptions on the operators $G_\s$:

\begin{equation}\label{c1}\tag{$A_1$}
  \|G_\s f\|_p\le \kappa_1\|f\|_{\ell_p(X_\s)},\quad f\in L_p(\R^d)\cap \ell_p(X_\s);
\end{equation}
\begin{equation}\label{c1'}\tag{$A_2$}
  \kappa_2\|f\|_{\ell_p(X_\s)}\le \|G_\s f\|_p,\quad f\in L_p(\R^d)\cap \ell_p(X_\s);
\end{equation}
\noindent
\begin{equation}\label{c2'}\tag{$A_3$}
\text{there exists}\,\, s\in \N\,\, \text{and}\,\, \l \in (0,1]\,\, \text{such that}
\\
  \phantom{\|g-G_\s g\|_p\le \kappa_3\s^{-s}|g|_{W_p^s}, \quad g\in \mathcal{B}_p^{\l\s}(\R^d),}
\end{equation}
$$
\|g-G_\s g\|_p\le \kappa_3\s^{-s}|g|_{W_p^s}, \quad g\in \mathcal{B}_p^{\l\s}(\R^d),
$$
where   
$\kappa_i=\kappa_i(p,d)$, $i=1,2$, and  $\kappa_3=\kappa_3(\l, s,p,d)$ are some positive constants.

Note that in assumption~\eqref{c2'}, we do not require $s>d/p$, as it is usual in the sampling and interpolation theory.

We formulate our main results in terms of moduli of smoothness and special averaged operators.
Recall that the (integral) modulus of smoothness of a function $f\in L_p(\R^d)$, $1\le p<\infty$,
is given by
\begin{equation}\label{eqmod1}
    \w_r(f,\d)_p=\sup_{|h|<\d} \Vert \D_h^r f\Vert_p,
\end{equation}
where
$$
\D_h^r f(x)=\sum_{\nu=0}^r\binom{r}{\nu}(-1)^{\nu} f(x+(r-\nu)h),
$$
$\binom{r}{\nu}=\frac{r (r-1)\dots (r-\nu+1)}{\nu!},\quad \binom{r}{0}=1$. 
For $f\in L_1^{\rm loc}(\R^d)$, $\d>0$, and $r\in \N$, we define the special averaged operator by
\begin{equation}\label{SB3}
  f_{\d,r}(x)=-\frac{2}{\binom{2r}{r}}\sum_{j=1}^r (-1)^j\binom{2r}{r-j}f_{\frac{\d j}{r}}(x),
\end{equation}
where
\begin{equation*}
  f_\d(x)=\frac1{|B_\d|}\int_{B_\d(x)} f(y)dy.
\end{equation*}
By the standard calculations, we have
\begin{equation}\label{SB1}
  f(x)-f_{\d,r}(x)=\frac1{c_r|B_1|}\int_{B_1} \D_{\frac{\d}{r}y}^{2r}f\(x\)dy,
\end{equation}
where $c_r=(-1)^{r+1}\binom{2r}{r}$ is a normalized constant.

Recall that
\begin{equation}\label{st1}
  C_2\w_{2r}(f,\d)_p\le \|f_{\d,r}-f\|_p\le C_1\w_{2r}(f,\d)_p,
\end{equation}
where $C_1$ and $C_2$ are some positive constants depending only on $r$ and $d$.
Here, the upper estimate follows directly from Minkowski's inequality. The lower estimate can be proved using the methods of Fourier multipliers, see, e.g.,~\cite[8.2.5]{TB}, \cite{K17}.

It is known that the modulus of smoothness~\eqref{eqmod1}, and therefore the quantity $\|f_{\d,r}-f\|_{p}$, are not suitable for obtaining sharp estimates of the $L_p$-error of approximation by sampling operators, see, e.g., \cite{WS03}, \cite{KL23}, and~\cite{BXZ92} for the corresponding result in the periodic case. However, if we replace the $L_p$-norm with the discrete semi-norm $\|\cdot \|_{\ell_p(X_\s)}$, then the expression $\|f_{\d,r}-f\|_{\ell_p(X_\s)}$ becomes the correct quantity for studying the approximation properties of sampling operators in $L_p$ spaces, see~\cite{KL23} for the case of univariate periodic functions.
In the next result, we compare the quantity $\|f_{\d,r}-f\|_{\ell_p(X_\s)}$ with the averaged moduli of smoothness ($\tau$-moduli), which are commonly used to investigate approximation properties of sampling operators and quadrature formulas (see, e.g.,~\cite{SP}, \cite{CZ99}, \cite{BBSV06}, \cite{BMSVV14}, \cite{H89}). Recall that the averaged modulus of smoothness is defined by
$$
\tau_r(f,\d)_p=\|\w_r(f,\cdot,\d)\|_p=\(\int_{\R^d} (\w_r(f,x,\d))^pdx\)^{1/p},
$$
where
$$
\w_r(f,x,\d)=\sup\left\{|\D_h^r f(t)|\,:\,t,t+rh\in B_{r\d/2}(x)\right\}
$$
is the local modulus of smoothness of $f$.

\begin{proposition}\label{cor1+}
Let $f$ be a bounded function in $L_p(\R^d)$, $1\le p<\infty$, $r\in \N$, and $0<\d\le \g/\s$.
Then
\begin{equation}\label{EQ0}
    \|f_{\d,r}-f\|_{\ell_p(X_\s)}\le \frac{C}{(\d \s)^{d/p}}\tau_{2r}(f,\d)_p.
 \end{equation}
In particular, if $f\in C(\R^d)$ and $2r>d/p$, then
\begin{equation}\label{tauint}
  \|f_{\d,r}-f\|_{\ell_p(X_\s)}\le \frac{C}{\s^{d/p}}\int_0^\d\frac{\omega_{2r}(f,t)_p}{t^{d/p}}\frac{dt}{t},
\end{equation}
where the constant $C$ depends only on $r$, $\g$, and $d$.
\end{proposition}

\begin{proof}
Let $\xi\in X_\s$.  It follows from~\eqref{SB1} that
\begin{equation}\label{EQ1}
\begin{split}
    |f_{\d,r}(\xi)-f(\xi)|&\le \frac1{c_r|B_1|}\int_{B_1}|\D_{\frac{\d y}r}^{2r}f(\xi)|dy.
\end{split}
\end{equation}
For any $\t\in B_{2\d}(\xi)$, we have $B_{2\d}(\xi)\subset B_{4\d}(\t)$. Thus, for all $y\in B_1$, we get the following estimates
\begin{equation*}
  \begin{split}
      |\D_{\frac{\d y}r}^{2r}f(\xi)|\le \w_{2r}(f,\xi,2\d/r)\le \w_{2r}(f,\t,4\d/r).
  \end{split}
\end{equation*}
Finally,  from~\eqref{EQ1}, we derive
\begin{equation*}
  \begin{split}
      \frac1{\s^d} \sum_{\xi\in X_\s}|f_{\d,r}(\xi)-f(\xi)|^p
      &\le \frac{C}{(\d \s)^{d}} \sum_{\xi \in X_\s}\int_{B_{4\d}(\t)}\(\w_{2r}(f,\t,4\d/r)\)^p\,d\t\\
&\le \frac{C}{(\d \s)^{d}}\tau_{2r}(f,4\d/r)_p^p,
  \end{split}
\end{equation*}
which, together with property $(c')$ below of averaged moduli of smoothness, implies~\eqref{EQ0}.

Inequality~\eqref{tauint} follows directly from~\eqref{EQ0} and estimate~\eqref{taum} given in the next section.
\end{proof}

\section{Auxiliary results}\label{secAR}
\subsection{Moduli of smoothness and the error of best approximation}
Recall several basic properties of the moduli of smoothness (see, e.g.,~\cite[Ch.~2]{DL}, \cite[Ch.~4]{TB} for the case $d=1$ and~\cite{KT21}, \cite{KT20_2} for $d\ge 2$).
For  $f,g\in L_p(\R^d)$, $1\le p<\infty$, $\d>0$, and $r\in \N$, we have
\begin{itemize}
\item[$(a)$]
$ \w_r(f,\d)_p$ is a non-negative non-decreasing function of $\d$ such that
$\lim\limits_{\d\to 0+} \w_r(f,\delta)_p=0$;
\item[$(b)$]
$\w_r(f+g,\d)_p\le \w_r(f,\d)_p+\w_r(g,\d)_p;$

\item[$(c)$]  $ \w_{r+1}(f,\d)_p\le 2\w_r(f,\d)_p$;

\item[$(d)$]
for $\l>0$,
        $$\w_r(f,\lambda \delta)_p\le (1+\l)^r  \w_r(f,\d)_p;$$

\item[$(e)$] $\w_r(f,\d)_p\le \d^r |f|_{W_p^r}$ for all $f\in W_p^r(\R^d)$.
\end{itemize}

We also need a well-known lemma about the equivalence of the modulus of smoothness and the corresponding $K$-functional, see~\cite[Ch.~5]{BeSh} and~\cite{JS77}.  Recall that the $K$-functional of the pair $(L_p(\R^d), W_p^s(\R^d))$ is defined by 
$$
K_s(f,\d)_p=\inf_{g\in W_p^s(\R^d)}\{\|f-g\|_p+\d^s |g|_{W_p^s(\R^d)}\}.
$$

\begin{lemma}\label{leKM}
  Let $f\in L_p(\R^d)$, $1\le p<\infty$, and $s\in \N$. Then
  \begin{equation}\label{eqvK}
 K_s(f,\d)_p\asymp \omega_s(f,\d)_p,\quad \d>0,
\end{equation}
where $\asymp$ is a two-sided inequality with positive constants independent of $f$ and $n$.
\end{lemma}


\smallskip

The following basic properties of the averaged moduli of smoothness can be found in~\cite[Ch.~1]{SP}, \cite{BMSVV14}, and~\cite{I86}.
For bounded functions  $f$ and $g$ in $L_p(\R^d)$, $1\le p<\infty$, $\d>0$, and $r\in \N$, we have
\begin{itemize}
\item[$(a')$]
$ \tau_r(f,\d)_p$ is a non-negative non-decreasing function of $\d$;

\item[$(b')$]
$\tau_r(f+g,\d)_p\le \tau_r(f,\d)_p+\tau_r(g,\d)_p;$

 \item[$(c')$]  $ \tau_{r+1}(f,\d)_p\le 2\tau_r\(f,\frac{r+1}{r}\d\)_p$;

\item[$(d')$]
for $\l>0$,
        $$
        \tau_r(f,\lambda \delta)_p\le C(r,\l)  \tau_r(f,\d)_p;
        $$
\item[$(e')$] $\tau_r(f,\d)_p\le C\bigg(\d^{r}|f|_{W_p^r}+\sum\limits_{r<|\b|_1\le d,\,\b\in \{0,1\}^d} \d^{|\b|_1} \|D^\b f\|_p\bigg)$,\\
provided the derivatives on the right-hand side exist as an element of $L_p(\R^d)$;

\item[$(f')$]  if $r>d/p$ and the right-hand side of~\eqref{taum} is finite, then $f$ coincides a.e. with a continuous function $F$ such that
\begin{equation}\label{taum}
  \tau_r(F,\d)_p\le C(r,d,p)\d^{d/p}\int_0^\d \frac{\omega_r(f,t)_p}{t^{d/p}}\frac{dt}{t}.
\end{equation}
\end{itemize}

\smallskip

In what follows, we assume that the family of subspaces $\Phi=(\Phi_\s)_{\s\ge 0}$ has the following standard properties (see~\cite[p.~217]{DL}): $0\in \Phi_\s$ for all $\s>0$, $\Phi_0=\{0\}$; $a\Phi_\s=\Phi_\s$ for each $a\neq 0$; $\Phi_\s+\Phi_\s \subset \Phi_{c\s}$ for some constant $c=c(\Phi)$. We also assume that for given parameters $p\in [1,\infty)$ and $s\in \N$, the family $\Phi=(\Phi_\s)_{\s\ge 0}$ satisfies
Jackson and Bernstein type inequalities. Namely, there exist positive constants $c_1$ and $c_2$ such that for all $\s>0$, we have
\begin{equation}\label{js}
  E(f,\Phi_\s)_p\le c_1\s^{-s}|f|_{W_p^s},\quad f\in W_p^s(\R^d),
\end{equation}
\begin{equation}\label{bs}
  |g|_{W_p^s}\le c_2 \s^s\|g\|_p,\quad g\in\Phi_\s,
\end{equation}
where $E(f,\Phi_\s)_p$ denotes the error of best approximation of $f$ by elements of $\Phi_\s$ in $L_p(\R^d)$:
$$
E(f,\Phi_\s)_p=\inf_{g \in \Phi_\s}\|f-g\|_p.
$$

%
%
In a natural way, inequalities~\eqref{js} and~\eqref{bs} imply the direct and inverse approximation theorems stated in terms of moduli of smoothness.

\begin{lemma}\label{leJB}
  Let  $s\in \N$ and $1\le p<\infty$.
   \begin{enumerate}
    \item[$1)$] If the Jackson inequality~\eqref{js} holds, then
     \begin{equation}\label{leJB.1}
          E(f,\Phi_\s)_p\le C\,\w_s(f,\s^{-1})_p, \quad f\in L_p(\R^d).
      \end{equation}
    \item[$2)$] If the Bernstein inequality~\eqref{bs} holds, then
      \begin{equation}\label{leJB.2}
         \w_s(f,\s^{-1})_p\le \frac{C}{\s^s}\sum_{\nu=0}^{[\s]}(\nu+1)^{s-1}E(f,\Phi_\nu)_p, \quad f\in L_p(\R^d).
      \end{equation}
  \end{enumerate}
  Here, the constant $C$ is independent of $f$ and $\s$.
\end{lemma}

\begin{proof}
The assertion of the lemma can be proved by repeating the proof of~\cite[Theorem~5.1, Ch.~7]{DL} and applying Lemma~\ref{leKM}, see also~\cite{HI90}.
\end{proof}

\begin{lemma}\label{lns}
  Let $1\le p<\infty$ and $s\in \N$. Assume that $\Phi_\s$ is such that, for each $\a\in \Z_+^d,\, |\a|_1=s$,
\begin{equation}\label{bebe2}
  \|\tfrac{\partial}{\partial x_i} D^{\a}g_\s\|_p\le B \s \| D^{\a}g_\s\|_p,\quad i=1,\dots,d,\quad g_\s\in \Phi_\s,
\end{equation}
for some constant $B$ independent of $\s$.
Then
\begin{equation}\label{ns}
  \s^{-s}|g_\s|_{W_p^s}\le C\omega_s(g_\s,\s^{-1})_p,\quad g_\s\in \Phi_\s,
\end{equation}
where the constant $C$ does not depend on $g_\s$.
\end{lemma}

\begin{proof}
First, we consider the one-dimensional case, $d=1$. Let $x\in \R$, $h>0$, and $m\le s$, $m\in \N$. Applying Taylor's formula, we obtain
\begin{equation*}
  \begin{split}
     \Delta_h^m g_\s(x)&=\sum_{\nu=0}^m(-1)^{\nu+s}\binom{m}{\nu}\bigg(\sum_{\a=0}^m\frac{g_\s^{(\a)}(x)}{\a!}(\nu h)^\a\\
     &\qquad+\frac{(\nu h)^{m+1}}{m!}\int_0^1 (1-t)^m g_\s^{(m+1)}(x+\nu th)dt\bigg)\\
     &=h^m g_\s^{(m)}(x)+\frac{h^{m+1}}{m!}\int_0^1 (1-t)^m \sum_{\nu=0}^m(-1)^\nu\binom{m}{\nu} \nu^{m+1}g_\s^{(m+1)}(x+\nu th)dt.
  \end{split}
\end{equation*}
Then, using the triangle inequality and Minkovski's inequality, we get
\begin{equation*}
  h^m\|g_\s^{(m)}\|_p\le \|\D_h^m g_\s\|_p+\frac{h^{m+1}}{(m+1)!}\sum_{\nu=0}^m \binom{m}{\nu} \nu^{m+1}\|g_\s^{(m+1)}\|_p.
\end{equation*}
This inequality together with assumption~\eqref{bebe2} yields
\begin{equation*}
  (1-cBh\s)h^m\|g_\s^{(m)}\|_p\le \|\D_h^m g_\s\|_p,
\end{equation*}
where the constant $c$ depends only on $s$. Thus, for any  $0<h<\frac1{2cB \s}$, we have
\begin{equation}\label{ns3}
  h^m\|g_\s^{(m)}\|_p \le C\|\D_h^m g_\s\|_p,
\end{equation}
which together with property $(d)$ of the modulus of smoothness implies~\eqref{ns}.

Now let $x\in \R^d$. Applying in each variable~\eqref{ns3}, we get, for all $\a\in \Z_+^d$, $|\a|_1=s$, the following inequalities
\begin{equation}\label{ns4}
\begin{split}
   \s^{-s}|g_\s|_{W_p^s}&=\sum_{|\a|_1=s}\s^{-\a_1}\dots \s^{-\a_d}\|D^\a g_\s\|_p\\
   &\le C\sum_{|\a|_1=s} \|\D_{\frac{ce_1}\s}^{\a_1}\dots \D_{\frac{ce_d}\s}^{\a_d}g_\s\|_p\\
   &\le C\sum_{|\a|_1=s} \sup_{|h|\le \frac c\s}\|\D_{e_1 h_1}^{\a_1}\dots \D_{e_d h_d}^{\a_d}g_\s\|_p\le C\omega_s(f,c\s^{-1})_p,
\end{split}
\end{equation}
where the latter inequality follows from~\cite[Lemma 4.11, p. 338]{BeSh}. 
Finally, \eqref{ns4} and property $(d)$ of moduli of moothness imply~\eqref{ns}.
\end{proof}

The next lemma can be proved by the standard scheme (see, e.g.,~\cite{R94} or~\cite{KL23}) by employing inequalities~\eqref{leJB.1} and~\eqref{leJB.2} and properties $(b)$, $(c)$, $(d)$ of moduli of smoothness.

\begin{lemma}\label{lemR} 
  Let $f\in L_p(\R^d)$, $1\le p<\infty$, and $r\in \N$. 
  If there exists a positive constant $F$ such that
$$
\w_r(f,\d)_p\le F\w_{r+1}(f,\d)_p,\quad \d>0,
$$
then 
$$
\w_r(f,\s^{-1})_p\le GE(f,\Phi_\s)_p
$$
for some constant $G$ independent of $\s$.
\end{lemma}

One of the important examples for $\Phi_\s$, where Lemmas~\ref{leJB}, \ref{lemR}, and~\ref{lns} hold, is the class of band-limited functions $\mathcal{B}_p^\s(\R^d)$. In the next two lemmas, we recall several classical inequalities from approximation theory for this case. Everywhere below we denote
$$
E_\s(f)_p=E(f,\mathcal{B}_p^\s(\R^d))_p.
$$

\begin{lemma}\label{leJB+} 
Let $f\in L_p(\R^d)$, $1\le p<\infty$, and $r\in \N$. Then
\begin{equation}\label{leJB.1+}
  E_\s(f)_p\le C\w_r(f,\s^{-1})_p
\end{equation}
and
\begin{equation}\label{leJB.2+}
  \w_r(f,1/\s)_p\le \frac{C}{\s^r}\sum_{\nu=0}^{[\s]}(\nu+1)^{r-1}E_\nu(f)_p,
\end{equation}
where the constant $C$ does not depend on $f$ and $\s$. Further, for all $g_\s\in \Phi_\s$, we have
\begin{equation}\label{NS+}
\d^{-r}\w_r (g_\s,\d)_{p}\asymp |g_\s|_{W_p^r},\quad 0<\delta\le \s^{-1}.
\end{equation}
In particular, if $g_\s\in \mathcal{B}_p^\s(\R^d)$ is such that $\|f-g_\s\|_p\le c\omega_r(f,\s^{-1})_p$, where the constant $c$ is independent of $f$ and $\s$, then
\begin{equation}\label{eqvR}
 \omega_r(f,\s^{-1})_p\asymp \|f-g_\s\|_p+\s^{-r}|g_\s|_{W_p^r}
\end{equation}
and
\begin{equation}\label{NS}
  |g_\s|_{W_p^r}\le C\s^r\w_r(f,\s^{-1})_p.
\end{equation}
In the above inequalities $C$ is some positive constant independent of $f$ and $\s$.
\end{lemma}

\begin{proof}
The relations~\eqref{leJB.1+}--\eqref{NS} are well-known, see~\cite{KT20_2} and~\cite{KT21}. See also~\cite[Ch.~5 and Ch.~6]{timan} for the classical  inequalities~\eqref{leJB.1+} and~\eqref{leJB.2+}.
\end{proof}

Recall also the classical Bernstein type inequalities, see, e.g.,~\cite[Ch.~2 and Ch.~3]{nikol-book}.
\begin{lemma}\label{leBer}
  Let $1\le p<\infty$, and $\a\in \Z_+^d$. Then, for each $g_\s\in \mathcal{B}_p^\s(\R^d)$,
  \begin{equation}\label{Ber}
    \| D^\a g_\s\|_p\le \s^{|\a|_1}\|g_\s\|_p.
  \end{equation}
\end{lemma}

\section{Main results}

\subsection{Direct and inverse inequalities}\label{S1}

\begin{theorem}\label{th1+}
  Let $f\in L_p(\R^d)$, $1\le p<\infty$, and $r,s\in \N$. Suppose $G_\s$ satisfies assumptions~\eqref{c1} and~\eqref{c2'} with $s\le 2r$. Then
  \begin{equation}\label{th1.1'}
     \|f-G_\s f\|_p\le \kappa_1\|f_{\g/\s,r}-f\|_{\ell_p(X_\s)}+{C_1}\w_s(f,\s^{-1})_p.
  \end{equation}
If, additionally, $f\in \ell_p(X_\s)$ and \eqref{c1'} holds, then
  \begin{equation}\label{th1.2'}
     \kappa_2\|f_{\g/\s,r}-f\|_{\ell_p(X_\s)}-{C_2}\w_s(f,\s^{-1})_p\le \|f-G_\s f\|_p.
  \end{equation}
Here, $C_1$ and $C_2$ are some positive constants independent of $f$ and $\s$.
\end{theorem}

Taking into account the fact that $\lim_{\d\to 0}\w_{s}(f,\d)_p=0$  for any $f\in L_p(\R^d)$, we obtain by Theorem~\ref{th1+} the following convergence criteria.

\begin{corollary}\label{th2}
  Let $f\in L_p(\R^d)\cap \ell_p(X_\s)$, $1\le p<\infty$, and $r,s\in \N$. Suppose that for all $\s>0$, the operator $G_\s$  satisfies assumptions~\eqref{c1}, \eqref{c1'}, and~\eqref{c2'} with $s\le 2r$. Then $G_\s f$ converges to  $f$ as $\s\to \infty$ in $L_p(\R^d)$ if and only if
  \begin{equation}\label{conv}
    \|f_{\g/\s,r}-f\|_{\ell_p(X_\s)}\to 0\quad\text{as}\quad \s\to \infty.
  \end{equation}
\end{corollary}


Using Theorem~\ref{th1+}, we determine the exact order of convergence of $G_\s f$ under specific assumptions on the function $f$.

\begin{corollary}\label{cor1}
 Let $f\in L_p(\R^d)\cap \ell_p(X_\s)$, $1\le p<\infty$, and $r,s\in \N$. Suppose that for all $\s>0$, the operator $G_\s$  satisfies assumptions~\eqref{c1}, \eqref{c1'}, and~\eqref{c2'} with $s\le 2r$. If
  $$
  \w_s(f,\s^{-1})_p=o\(\|f_{\g/\s,\,r}-f\|_{\ell_p(X_\s)}\),
  $$
  then
  $$
  \|f-G_\s f\|_p\sim \|f_{\g/\s,\,r}-f\|_{\ell_p(X_\s)}.
  $$
\end{corollary}

The following result presents an analogue of Bernstein’s-type inverse inequality~\eqref{leJB.2} in the context of sampling operators.

\begin{theorem}\label{th3}
  Let $f\in L_p(\R^d)\cap \ell_p(X_\s)$, $1\le p<\infty$, and $r,s\in \N$. Suppose that  for all $\s>0$, the operator $G_\s$  satisfies assumptions~\eqref{c1}, \eqref{c1'}, and~\eqref{c2'} with $s\le 2r$ and that $G_\s f\in \Phi_\s$.
  Then
  \begin{equation}\label{th3.1}
  \begin{split}
         \|f_{\g/\s,\,r}-f\|_{\ell_p(X_\s)}&+\w_s(f,\s^{-1})_p\\
         &\le C\bigg(\|f-G_\s f\|_p+\frac1{\s^s}\sum_{\nu=0}^{[\s]}(\nu+1)^{s-1} \|f-G_\nu f\|_p\bigg).
  \end{split}
  \end{equation}
  where the constant $C$ does not depend on $f$ and $\s$.
\end{theorem}


The application of Theorems~\ref{th1+} and~\ref{th3} leads to the corollary below. 
\begin{corollary}\label{cor2}
Let $f\in L_p(\R^d)\cap \ell_p(X_\s)$, $1\le p<\infty$, $r,s\in \N$, $s\le 2r$, and $\a\in (0,s)$.  Suppose that for all $\s>0$, the operator $G_\s$ satisfies the conditions of Theorem~\ref{th3}. Then the following properties are equivalent:
    \begin{enumerate}
    \item[$(i)$]  $\|f-G_\s f\|_p=\mathcal{O}(\s^{-\a})$,
    \item[$(ii)$]  $\|f_{\g/\s,\,r}-f\|_{\ell_p(X_\s)}+\w_s(f,\s^{-1})_p=\mathcal{O}(\s^{-\a})$.
  \end{enumerate}
\end{corollary}

The next result complements Corollary~\ref{cor1} by providing more precise sufficient conditions for determining the exact rate of convergence of $G_\s f$ in $L_p(\R^d)$.

\begin{theorem}\label{th1RG}
  Let $f\in L_p(\R^d)\cap \ell_p(X_\s)$, $1\le p<\infty$, and $r,s\in\N$. Suppose that for all $\s>0$, the operator $G_\s$ satisfies the conditions of Theorem~\ref{th3}.  If there exists a constant $K>0$ such that
\begin{equation}\label{r1}
\w_s(f,\d)_p\le K\w_{s+1}(f,\d)_p,\quad \d>0,
\end{equation}
then
  \begin{equation*}
     \|f-G_\s f\|_p\asymp\|f_{\g/\s,\,r}-f\|_{\ell_p(X_\s)}+\w_s(f,\s^{-1})_p,
  \end{equation*}
where $\asymp$ is a two-sided inequality with positive constants independent of $\s$.
\end{theorem}


\subsection{$K$-functionals and their realizations}


In this section, we establish an analogue of equivalence~\eqref{eqvK} for the "semi-discrete" modification of the Peetre $K$-functional
\begin{equation*}
  \mathcal{K}_s(f,X_\s)_p=\inf_{g\in W_p^s(\R^d)}\(\|f-g\|_{\ell_p(X_\s)}+\|f-g\|_p+\s^{-s}|g|_{W_p^s}\).
\end{equation*}

\begin{theorem}\label{thKw}
  Let $f\in L_p(\R^d)\cap \ell_p(X_\s)$, $1\le p<\infty$, $r,s\in \N$, $s\le 2r$, and $s>d/p$.
  Then
  \begin{equation}\label{thKw.1}
     \mathcal{K}_s(f,X_\s)_p\asymp \|f_{\g/\s,\,r}-f\|_{\ell_p(X_\s)}+\w_s(f,\s^{-1})_p,
  \end{equation}
where $\asymp$ is a two-sided inequality with positive constants independent of $f$ and $\s$.
\end{theorem}

It is known (see, e.g.,~\cite{HI90}) that the modulus of smoothness is equivalent to the so-called realization of the $K$-functional:
\begin{equation}\label{Re}
  \w_s(f,\s^{-1})_p\asymp \|f-P_\s f\|_p+\s^{-s}|P_\s f|_{W_p^s},\quad f\in L_p(\R^d),
\end{equation}
where $P_\s f$ is a suitable approximation of a function $f$ satisfying the Jackson-type inequality $\|f-P_\s f\|_p\le c(s,p,d)\w_s(f,\s^{-1})_p$, and $\asymp$ denotes a two-sided inequality with positive constants independent of $f$ and $\s$.
For various applications of realizations of $K$-functionals see e.g.,~\cite{HI90}, \cite{KT20}--\cite{KT21}.
It is clear that~\eqref{Re} does not generally hold for sampling operators.  However, by incorporating the quantity $\|f_{\g/\s,\,r}-f\|_{\ell_p(X_\s)}$ into the corresponding modulus of smoothness, we can establish an analogue of the equivalence~\eqref{Re} for certain operators $G_\s$.


\begin{theorem}\label{th4} Let $f\in L_p(\R^d)\cap \ell_p(X_\s)$, $1\le p<\infty$, and $r,s\in \N$. Suppose that $G_\s$ satisfies conditions~\eqref{c1}, \eqref{c1'}, and
\begin{equation*}
  \|g-G_\s g\|_p\le \kappa_3\s^{-s}|g|_{W_p^s}, \quad g\in W_p^s(\R^d).
\end{equation*}
Suppose also that $G_\s f\in \Phi_\s\cap W_p^s(\R^d)$ and that the Bernstein-type inequality~\eqref{bebe2} holds for elements of~$\Phi_\s$.
Then
\begin{equation}\label{th4.1}
  \|f_{\g/\s,\,r}-f\|_{\ell_p(X_\s)}+\w_s(f,\s^{-1})_p\asymp \|f-G_\s f\|_p+\s^{-s}|G_\s f|_{W_p^s},
\end{equation}
where $\asymp$ is a two-sided inequality with positive constants independent of $f$ and $\s$.
\end{theorem}

\subsection{Smoothness of sampling operators}\label{S3}
Now we consider certain interrelations between moduli of
smoothness of functions  and smoothness properties
of sampling operators. Equivalence~\eqref{th4.1} implies that
\begin{equation*}
 \s^{-s}|G_\s f|_{W_p^s} \le C\(\|f_{\g/\s,r}-f\|_{\ell_p(X_\s)}+\w_s(f,\s^{-1})_p\),
\end{equation*}
where the constant $C$ is independent of $f$ and $\s$.
In the next theorem, we establish a corresponding inverse estimate.

\begin{theorem}\label{th5}
Let $f\in L_p(\R^d)\cap \ell_p(X_\s)$, $1\le p<\infty$, and let $f$ satisfy~\eqref{conv} for given $r,s\in \N$. Suppose that for all $\s>0$, the operator  $G_\s$  satisfies conditions~\eqref{c1}, \eqref{c1'}, \eqref{c2'} with $s\le 2r$, and that $G_\s f\in \Phi_\s \cap W_p^s(\R^d)$.
Suppose also that $G_\s f(\xi)=f(\xi)$ for all $\xi\in X_\s$, and $X_\s\subset X_{2\s}$. Then
\begin{equation}\label{th5.1}
  \|f_{\g/\s,\,r}-f\|_{\ell_p(X_\s)}+\w_s(f,\s^{-1})_p\le C\sum_{k=0}^\infty (\s2^k)^{-s}|G_{2^k \s}f|_{W_p^s},
\end{equation}
where the constant $C$ is independent of  $f$ and $\s$.
\end{theorem}

Using Theorems~\ref{th5} and~\ref{th4}, we refine Corollary~\ref{cor2} as follows.


\begin{corollary}\label{cor3}
  Let $f\in L_p(\R^d)\cap \ell_p(X_\s)$, $1\le p<\infty$, $r,s\in \N$, $s\le 2r$, and $\a\in (0,s)$. Assume that for all $\s>0$, the operator  $G_\s$ satisfies the conditions of Theorem~\ref{th5}. Then the following properties are equivalent:
  \begin{enumerate}
    \item[$(i)$]  $\|f-G_\s f\|_p=\mathcal{O}(\s^{-\a})$,
    \item[$(ii)$]  $\|f_{\g/\s,\,r}-f\|_{\ell_p(X_\s)}+\w_s(f,\s^{-1})_p=\mathcal{O}(\s^{-\a})$,
    \item[$(iii)$] $|G_\s f|_{W_p^s}=\mathcal{O}(\s^{s-\a})$.
  \end{enumerate}
If, additionally, $s>d/p$, then the above properties are equivalent to
  \begin{enumerate}
    \item[$(iv)$]  $\mathcal{K}_s(f,X_\s)_p=\mathcal{O}(\s^{-\a})$.
  \end{enumerate}
\end{corollary}

\subsection{More general measures of smoothness}

In this section, we extend the results of Sections~\ref{S1}--\ref{S3} to more general measures of smoothness.
To this end, we consider a general translation-invariant smooth Banach space $Y_p^s(\R^d)$, $1\le p<\infty$ and $s>0$, endowed with the semi-norm $|\cdot|_{Y_p^s}$. Typical examples are the Besov and Triebel-Lizorkin spaces $F_{p,q}^s(\R^d)$, see, e.g.,~\cite{ST86}, \cite{SS00}.

Recall that the $K$-functional of the pair $(L_p, Y_p^s)$ is defined by
$$
K(f,t,L_p, Y_p^s)=\inf_{g\in Y_p^s}\(\|f-g\|_p+t|g|_{Y_p^s}\).
$$
Note that for all $s\in \N$ and $1\le p<\infty$, we have $K_s(f,\d)_p=K(f,\d,L_p, W_p^s)$.

In what follows, we assume that the semi-norm $|\cdot|_{Y_p^s}$ satisfies the following properties:
\begin{itemize}
  \item[-] for given $s>0$ and $r\in \N$ with $0<s\le 2r$ and $1\le p<\infty$, there exists a constant $C$ such that
    \begin{equation*}
      |g|_{W_p^{2r}}\le C \s^{2r-s}|g|_{Y_p^{s}},\quad g\in \mathcal{B}_p^\s(\R^d);
    \end{equation*}
  \item[-]  for given  $s>0$ and $1\le p<\infty$, there exists a constant $C$ such that
     \begin{equation}\label{prY2}
      |g|_{Y_p^s}\le C\s^s K(g,\s^{-s} ,L_p, Y_p^s),\quad g\in \mathcal{B}_p^\s(\R^d).
    \end{equation}
\end{itemize}

Let $g_\s\in \mathcal{B}_p^\s(\R^d)$ be such that $\|f-g_\s\|_p\le c E_\s(f)_p$ with some constant $c$ independent of~$f$ and $\s$. It follows from~\eqref{prY2}, see, e.g.,~\cite{HI90}, that
\begin{equation}\label{prY3}
  K(f,t,L_p, Y_p^s)\asymp \|f-g_\s\|_p+t|g_\s|_{Y_p^s},\quad t>0, \quad f\in L_p(\R^d),
\end{equation}
where $\asymp$ is a two-sided inequality with positive constants independent of $f$ and $\s$.
For our purposes, we need also the following generalization of assumption~\eqref{c2'}:
\begin{equation}\label{c2'Y}\tag{$A_3'$}
\text{there exists}\,\, s\in \N\,\, \text{and}\,\, \l \in (0,1]\,\, \text{such that}
\\
  \phantom{\|\vp-G_\s \vp\|_p\le \kappa_3\s^{-s}|\vp|_{Y_p^s}, \quad \vp\in \mathcal{B}_p^{\l\s}(\R^d),}
\end{equation}
$$
\|\vp-G_\s \vp\|_p\le \kappa_3\s^{-s}|\vp|_{Y_p^s}, \quad \vp\in \mathcal{B}_p^{\l\s}(\R^d),
$$
where $\kappa_3=\kappa_3(\l, s,p,d)$ is a positive constant.

\smallskip

The following two theorems can be proved repeating step by step the proofs of Theorems~\ref{th1+} and~\ref{th3}, respectively

\begin{theorem}\label{th1+Y}
  Let $f\in L_p(\R^d)$, $1\le p<\infty$, and $r\in \N$. Suppose $G_\s$ satisfies~\eqref{c1} and~\eqref{c2'Y} with $0<s\le 2r$. Then
  \begin{equation*}
     \|f-G_\s f\|_p\le \kappa_1\|f_{\g/\s,r}-f\|_{\ell_p(X_\s)}+{C_1}K(f,\s^{-s},L_p, Y_p^s).
  \end{equation*}
If, additionally, $f\in \ell_p(X_\s)$ and~\eqref{c1'} holds, then
  \begin{equation}\label{th1.2'Y}
     \kappa_2\|f_{\g/\s,r}-f\|_{\ell_p(X_\s)}-{C_2}K(f,\s^{-s},L_p, Y_p^s)\le \|f-G_\s f\|_p.
  \end{equation}
Here, $C_1$ and $C_2$ are  positive constants independent of $f$ and $\s$.
\end{theorem}

\begin{theorem}\label{th3Y}
  Let $f\in L_p(\R^d)\cap \ell_p(X_\s)$, $1\le p<\infty$, and $r,n\in \N$. Suppose that $G_\s$ satisfies conditions~\eqref{c1}, \eqref{c1'}, and~\eqref{c2'Y} with $0<s\le 2r$, and $G_\s f\in \Phi_\s$. Suppose also that for elements of $\Phi_\s$, the following Bernstein inequality inequality is valid:
  \begin{equation}\label{bsY}
  |g|_{Y_p^s}\le c \s^s\|g\|_p,\quad g\in\Phi_\s,
\end{equation}
for some constant $c$ independent of $g$ and $\s$.
   Then
  \begin{equation*}
  \begin{split}
         \|f_{\g/\s,\,r}-f\|_{\ell_p(X_\s)}&+K(f,\s^{-s},L_p, Y_p^s)\\
         &\le C\bigg(\|f-G_\s f\|_p+\frac1{\s^s}\sum_{\nu=0}^{[\s]}(\nu+1)^{s-1} \|f-G_\nu f\|_p\bigg).
  \end{split}
  \end{equation*}
  where the constant $C$ is independent of $f$ and $\s$.
\end{theorem}

\section{Proofs of the main results}

\begin{proof}[Proof of Theorem~\ref{th1+}]
Denote $\d={\g}{\s^{-1}}$.
We have 
\begin{equation}\label{th1.3'}
\begin{split}
    \|f-G_\s f\|_p&\le \|f-f_{\d,r}\|_p+\|f_{\d,r}-G_\s(f_{\d,r})\|_p+\|G_\s(f_{\d,r}-f)\|_p.\\
\end{split}
\end{equation}
By~\eqref{st1} and property $(c)$, we get
\begin{equation}\label{th1.3'X}
\begin{split}
    \|f-f_{\d,r}\|_p\le C\w_{2r}(f,\d)_p\le C\w_{s}(f,\s^{-1})_p.
\end{split}
\end{equation}
Next, assumption~\eqref{c1} yields
\begin{equation}\label{th1.3'XX}
\begin{split}
    \|G_\s(f_{\d,r}-f)\|_p\le \kappa_1\|f-f_{\d,r}\|_{\ell_p(X_\s)}.
\end{split}
\end{equation}
To estimate $\|f_{\d,r}-G_\s(f_{\d,r})\|_p$, we take $g_\s\in \mathcal{B}_p^{\l\s}(\R^d)$ such that $\|f-g_\s\|_p\le CE_{\l\s}(f)_p$. Using Minkowski's inequality and~\eqref{c1} and~\eqref{c2'}, we obtain
\begin{equation}\label{th1.10'}
\begin{split}
&\|f_{\d,r}-G_\s(f_{\d,r})\|_p\\
&\le \|f_{\d,r}-(g_\s)_{\d,r}\|_p+\|(g_\s)_{\d,r}-G_\s((g_\s)_{\d,r})\|_p+\|G_\s(f_{\d,r}-(g_\s)_{\d,r})\|_p\\
&\le C\|f-g_\s\|_p+C\kappa_3\s^{-s}|g_\s|_{W_p^s}+\kappa_1\|f_{\d,r}-(g_\s)_{\d,r}\|_{\ell_p(X_\s)}.
\end{split}
\end{equation}
Now, taking into account that $B_\d(\xi)\cap B_\d(\xi')=\varnothing$ for $\xi,\xi' \in X_\s$, $\xi\neq \xi'$, and using H\"older's inequality,  we derive
\begin{equation}\label{th1.5'}
  \begin{split}
      \|f_{\d,r}-(g_\s)_{\d,r}\|_{\ell_p(X_\s)}^p
&\le \frac1{\s^d}\sum_{\xi\in X_\s}\(\sum_{j=1}^{r}\frac2{{|B_{\frac{\d j}{r}}|}}\int_{B_{\frac{\d j}{r}}(\xi)}|f(y)-g_\s(y)|dy\)^p\\
&\le \frac1{\s^d}\sum_{\xi\in X_\s}\(\sum_{j=1}^{r}\frac2{{|B_{\frac{\d j}{r}}|^{1/p}}}\(\int_{B_{\frac{\d j}{r}}(\xi)}|f(y)-g_\s(y)|^pdy\)^{1/p}\)^p\\
&\le\frac{C}{(\s\d)^d}
\sum_{\xi\in X_\s}\int_{B_{{\d}}(\xi)}|f(y)-g_\s(y)|^pdy\\
&\le \frac{C}{\g^d}\|f-g_\s\|_p^p.
  \end{split}
\end{equation}
Inequalities~\eqref{th1.10'} and \eqref{th1.5'} together with~\eqref{eqvR} give
\begin{equation}\label{th1.3'XX+}
\begin{split}
   \|f_{\d,r}-G_\s(f_{\d,r})\|_p &\le {C}{\g^{-d/p}}(\|f-g_\s\|_p+|g_\s|_{W_p^s})\\
   &\le {C}{\g^{-d/p}}\w_s(f,\s^{-1})_p.
\end{split}
\end{equation}
Thus, combining~\eqref{th1.3'}, \eqref{th1.3'X},  \eqref{th1.3'XX}, and~\eqref{th1.3'XX+}, we obtain~\eqref{th1.1'}.

\smallskip

Now we prove~\eqref{th1.2'}. Assumption~\eqref{c1'}, inequalities~\eqref{th1.3'X} and~\eqref{th1.3'XX+} yield
\begin{equation*}
\begin{split}
    \kappa_2\|f_{\d,r}-f\|_{\ell_p(X_\s)}&\le \|G_\s(f_{\d,r}-f)\|_p\\
  &\le \|f-f_{\d,r}\|_p+\|f_{\d,r}-G_\s(f_{\d,r})\|_p+\|f-G_\s f\|_p\\
  &\le {C}{\g^{-d/p}}\w_s(f,\s^{-1})_p+\|f-G_\s f\|_p,
\end{split}
\end{equation*}
which implies~\eqref{th1.2'}.
\end{proof}

%
%
%

\begin{proof}[Proof of Theorem~\ref{th3}]
From relation~\eqref{th1.2'}, we derive
\begin{equation}\label{th3.2}
\begin{split}
    \|f_{\g/\s,\,r}-f\|_{\ell_p(X_\s)}&+\w_s(f,\s^{-1})_p\\
  &\le \frac1{\kappa_2}\(\|f-G_\s f\|_p+(\kappa_2+C_2)\w_s(f,\s^{-1})_p\).
\end{split}
\end{equation}
This inequality together with inverse estimate~\eqref{leJB.2} and the fact that
\begin{equation}\label{EG}
  E(f,\Phi_\s)_p\le \|f-G_\s f\|_p
\end{equation}
gives~\eqref{th3.1}.
\end{proof}

\begin{proof}[Proof of Theorem~\ref{th1RG}]
In view of inequality~\eqref{th1.1'}, it is enough to establish that
$$
\|f_{\g/\s,\,r}-f\|_{\ell_p(X_\s)}+\w_s(f,\s^{-1})_p\le C\|f-G_\s f\|_p.
$$
By~\eqref{r1} and Lemma~\ref{lemR}, we have
\begin{equation*}
  \w_s(f,\s^{-1})_p\le CE(f,\Phi_\s)_p,
\end{equation*}
 which together with~\eqref{th3.2} and~\eqref{EG} proves the theorem.
\end{proof}

\begin{proof}[Proof of Theorem~\ref{thKw}]
  Let $g_\s\in \mathcal{B}_p^\s(\R^d)$ be such that $\|f-g_\s\|_p\le CE_\s(f)_p$. By the definition of $\mathcal{K}_s(f,X_\s)_p$, we have
\begin{equation}\label{Keq1}
\begin{split}
    \mathcal{K}_s(f,X_\s)_p&\le \|f-g_\s\|_{\ell_p(X_\s)}+\|f-g_\s\|_p+\s^{-s}|g_\s|_{W_p^s}.
\end{split}
\end{equation}
Next, for $\d=\g/\s$, we have
\begin{equation}\label{th1.4'}
  \begin{split}
      \|f-g_\s\|_{\ell_p(X_\s)}\le \|f-f_{\d,r}\|_{\ell_p(X_\s)}+\|f_{\d,r}-(g_\s)_{\d,r}\|_{\ell_p(X_\s)}+\|(g_\s)_{\d,r}-g_\s\|_{\ell_p(X_\s)}.
  \end{split}
\end{equation}
By~\eqref{EQ0}, properties of $\tau$-moduli of smoothness, and the Bernstein inequality~\eqref{Ber}, we get
\begin{equation}\label{th1.4'XX}
\begin{split}
    \|(g_\s)_{\d,r}-g_\s\|_{\ell_p(X_\s)}&\le C\tau_{2r}(g_\s,\s^{-1})_p\le C\tau_{s}(g_\s,\s^{-1})_p\\
    &\le C\(\s^{-s}|g_\s|_{W_p^s}+\sum_{s<|\b|_1\le d,\,\b\in \{0,1\}^d} \s^{-|\b|_1} \|D^\b g_\s\|_p\)\\
    &\le C\s^{-s}|g_\s|_{W_p^s}.
\end{split}
\end{equation}
Thus, combining~\eqref{Keq1}, \eqref{th1.4'}, and~\eqref{th1.4'XX}, and using~\eqref{th1.5'} and equivalence~\eqref{eqvR}, we obtain
\begin{equation*}
\begin{split}
    \mathcal{K}_s(f,X_\s)_p&\le \|f-f_{\d,r}\|_{\ell_p(X_\s)}+C(\|f-g_\s\|_p+\s^{-s}|g_\s|_{W_p^s})\\
&\le \|f_{\d,r}-f\|_{\ell_p(X_\s)}+C\w_s\(f,\s^{-1}\)_p,
\end{split}
\end{equation*}
which implies the upper estimate in~\eqref{thKw.1}.

\smallskip

Now prove the lower estimate. For any $g\in W_p^s(\R^d)$, we have
\begin{equation}\label{thKw.3}
  \begin{split}
     \|f_{\d,r}-f\|_{\ell_p(X_\s)}\le \|f_{\d,r}-g_{\d,r}\|_{\ell_p(X_\s)}+\|g_{\d,r}-g\|_{\ell_p(X_\s)}+\|g-f\|_{\ell_p(X_\s)}.
  \end{split}
\end{equation}
By the same arguments as in~\eqref{th1.5'}, we get
\begin{equation}\label{thKw.4}
  \begin{split}
      \|f_{\d,r}-g_{\d,r}\|_{\ell_p(X_\s)}\le {C}{\g^{-d/p}}\|f-g\|_p.
  \end{split}
\end{equation}
Further, applying~\eqref{EQ0}, \eqref{taum}, and properties $(c)$ and $(e)$, we have
\begin{equation}\label{thKw.4+}
  \begin{split}
      \|g_{\d,r}-g\|_{\ell_p(X_\s)}&\le C \tau_{2r}(g,\d)_p\le C\d^{d/p}\int_0^\d \frac{\omega_{2r}(g,t)_p}{t^{d/p}}\frac{dt}{t}\\
      &\le C\d^{d/p}\int_0^\d \frac{\omega_{s}(g,t)_p}{t^{d/p}}\frac{dt}{t}\le C\d^s |g|_{W_p^s},
  \end{split}
\end{equation}
which together with~\eqref{thKw.3} and~\eqref{thKw.4} implies
\begin{equation*}
  \begin{split}
    \|f_{\d,r}-f\|_{\ell_p(X_\s)}\le \|g-f\|_{\ell_p(X_\s)}+C{\g^{-d/p}}\|f-g\|_p+C\d^s |g|_{W_p^s}.
  \end{split}
\end{equation*}
Combining this estimate with~\eqref{eqvK}, we get
$$
\|f_{\d,r}-f\|_{\ell_p(X_\s)}+\w_s(f,\s^{-1})_p\le C\(\|g-f\|_{\ell_p(X_\s)}+\|f-g\|_p+\s^{-s}|g|_{W_p^s}\).
$$
It remains to take the infimum over all $g\in W_p^s(\R^d)$, completing the proof.
\end{proof}

\begin{remark}\label{remKK}
  In the statement of Theorem~\ref{thKw}, we can replace the assumption $s>d/p$ by the following one: for given $s\in \N$, there exists an operator $G_\s$ satisfying~\eqref{c1}, \eqref{c1'}, and  \eqref{c2'}.
Indeed, the only place in the proof of Theorem~\ref{thKw}, where we used $s>d/p$ is estimate~\eqref{thKw.4+}. An alternative way to obtain this inequality is based on the application of~Minkowskii's inequality, \eqref{st1}, and properties $(c)$ and $(e)$, which yield the following estimates:
\begin{equation*}
  \begin{split}
    \kappa_2\|g_{\d,r}-g\|_{\ell_p(X_\s)}&\le \|G_\s(g_{\d,r}-g)\|_p\\
    &\le \|G_\s(g_{\d,r}-g)-(g_{\d,r}-g)\|_p+\|g_{\d,r}-g\|_p\\
    &\le C\s^{-s}\|g_{\d,r}-g\|_{W_p^s}+C\w_{2r}(g,\d)_p\\
    &\le 2C\s^{-s}|g|_{W_p^s}+C\d^s|g|_{W_p^s}\\
    &=C\d^s |g|_{W_p^s}.
  \end{split}
\end{equation*}
\end{remark}

\begin{proof}[Proof of Theorem~\ref{th4}]
By~\eqref{eqvK}, we have
  \begin{equation*}
  \begin{split}
\w_s(f,\s^{-1})_p\le C\(\|f-G_{\s}f\|_p+\s^{-s}|G_{\s}f|_{W_p^s}\),
  \end{split}
  \end{equation*}
which together with~\eqref{th3.2} implies the estimate  from above in~\eqref{th4.1}.

Let us prove the estimate from below. By Lemma~\ref{lns} and properties $(b)$, $(c)$, we have
  \begin{equation}\label{th4.4}
    \begin{split}
      \s^{-s}|G_\s f|_{W_p^s}&\le C\w_s\(G_\s f,\s^{-1}\)_p\\
      &\le C\(2^s\|f-G_\s f\|_p+\w_s(f,\s^{-1})_p\).
    \end{split}
  \end{equation}
  Now, applying~\eqref{th1.1'} with $\d=\g/\s$ and~\eqref{th4.4}, we obtain
    \begin{equation*}
    \begin{split}
       \|f-G_\s f\|_p&+\s^{-s}|G_\s f|_{W_p^s}\\
       &\le (1+2^sC)\|f-G_\s f\|_p+C\w_s\(f,\s^{-1}\)_p\\
       &\le C\(\|f_{\g/\s,\,r}-f\|_{\ell_p(X_\s)}+\w_s\(f,\s^{-1}\)_p\),
    \end{split}
  \end{equation*}
  completing the proof.
\end{proof}

\begin{proof}[Proof of Theorem~\ref{th5}]
The theorem can be proved repeating step by step the proof of Theorem~4.6 in~\cite{KL23}, see also~\cite[Lemma~8]{HL} and~\cite{KT20}.
However, for the sake of completeness, we give the complete proof.

Applying~\eqref{th3.2} and~\eqref{eqvK}, we obtain
\begin{equation}\label{th5.4}
\begin{split}
    \|f_{\g/\s,\,r}-&f\|_{\ell_p(X_\s)}+\w_s(f,\s^{-1})_p\le C\(\|f-G_{\s}f\|_p+\s^{-s}|G_{\s}f|_{W_p^s}\).
\end{split}
\end{equation}
Next, by assumption~\eqref{c2'}, we have
\begin{equation}\label{th5.5}
I_\nu:=\|G_{2\nu}f-G_\nu(G_{2\nu}f)\|_p\le \kappa_3\nu^{-s}|G_{2\nu}f|_{W_p^s}.
\end{equation}
At the same time, taking into account that $X_\nu\subset X_{2\nu}$ and $G_\nu f(\xi)=f(\xi)$, $\xi\in X_\nu$, we get
\begin{equation*}
  \begin{split}
     I_\nu=\|G_{2\nu}f-G_\nu f\|_p\ge \|f-G_{\nu}f\|_p-\|f-G_{2\nu}f\|_p.
  \end{split}
\end{equation*}
Hence, in view of the convergence of $G_\nu f$ as $\nu\to\infty$ and~\eqref{th5.5}, we obtain
\begin{equation}\label{th5.6}
  \begin{split}
     \|f-G_{\s}f\|_p&=\sum_{k=0}^\infty \(\|f-G_{\s2^k}f\|_p-\|f-G_{\s2^{k+1}}f\|_p\)\\
     &\le \sum_{k=0}^\infty I_{\s2^k}\le C\sum_{k=1}^\infty (\s2^k)^{-s}|G_{\s2^k}f|_{W_p^s}.
  \end{split}
\end{equation}
Finally, combining~\eqref{th5.4} and~\eqref{th5.6}, we arrive at~\eqref{th5.1}.
\end{proof}

\begin{proof}[Proof of Theorem~\ref{th1+Y}]
  Repeat the proof of Theorem~\ref{th1+}, using relation~\eqref{prY3}.
\end{proof}

\begin{proof}[Proof of Theorem~\ref{th3Y}]
 Similarly as in the proof of Theorem~\ref{th3}, relation~\eqref{th1.2'Y} yields
\begin{equation}\label{th3.2Y}
\begin{split}
    \|f_{\g/\s,\,r}-f\|_{\ell_p(X_\s)}&+K(f,\s^{-s},L_p, Y_p^s)\\
  &\le \frac1{\kappa_2}\(\|f-G_\s f\|_p+(\kappa_2+C_2)K(f,\s^{-s},L_p, Y_p^s)\).
\end{split}
\end{equation}
By~\cite[\S~7, Theorem~5.1 (ii)]{DL}, we have that Bernstein inequality~\eqref{bsY} implies
      \begin{equation}\label{leJB.2Y}
         K(f,\s^{-s},L_p, Y_p^s)\le \frac{C}{\s^s}\sum_{\nu=0}^{[\s]}(\nu+1)^{s-1}E(f,\Phi_\nu)_p, \quad f\in L_p(\R^d).
      \end{equation}
Now, combining~\eqref{th3.2Y} and~\eqref{leJB.2Y} and using the inequality $E(f,\Phi_\s)_p\le \|f-G_\s f\|_p$, we prove the theorem.
\end{proof}

\section{Examples}

In this section, we demonstrate how the above results
can be applied to certain generalized sampling operators
\begin{equation}\label{SO}
S_\s^\vp f(x)=\sum_{k\in\Z^d}f\(\s^{-1}k\)\vp(\s x-k),\quad \s>0,
\end{equation}
where $\vp$ is a particular function. Below, we provide conditions on the function  $\vp$ that ensure the validity of assumptions~\eqref{c1}, \eqref{c1'}, and~\eqref{c2'} for $S_\s^\vp$.

\begin{proposition}\label{propA1A2} {\sc (See~\cite{JM91}.)}
 Let $\vp\in L_1(\R^d)$ and $1\le p<\infty$ be such that
\begin{equation}\label{jia1}
  \sum_{k\in \Z^d} |\vp(x+k)| \in L_p(0,1)^d.
\end{equation}
Then, for all $f\in \ell_p(X_\s)$, we have
 \begin{equation*}
   \|S_\s^\vp f\|_p\le C_1\|f\|_{\ell_p(X_\s)}.
 \end{equation*}
If, additionally, 
\begin{equation}\label{jia2}
  \sup_{k\in \Z^d} |\widehat{\vp}(\xi+k)|>0\quad\text{for all}\quad \xi\in \R^d,
\end{equation}
then
 \begin{equation*}
   C_2\|f\|_{\ell_p(X_\s)}\le \|S_\s^\vp f\|_p.
 \end{equation*}
Here, $C_1$ and $C_2$ are some positive constants independent of $f$ and $\s$.
\end{proposition}

Proposition~\ref{propA1A2} implies that if $\vp$ is such that~\eqref{jia1} and~\eqref{jia2} hold, then the operator $S_\s^\vp$ satisfies  assumptions~\eqref{c1} and~\eqref{c1'}. Concerning~\eqref{c2'} and~\eqref{c2'Y}, there are various approaches to ensure the required inequalities, see, e.g., \cite{KS20}, \cite{KS21}, \cite{KS21a}, \cite{SS00}. Below, we recall some standard results in this direction when
$\supp \widehat{\vp}$ is compact.

\begin{proposition}\label{propA3} 
Let $\vp\in \mathcal{B}_p^{a\s}(\R^d)$ and $g\in\mathcal{B}_q^{(1-a)\s}(\R^d)$ for some $a\in (0,1)$ and $1/p+1/q=1$, $1\le p<\infty$. Then
\begin{equation}\label{conv1}
  \s^{-d}\sum_{k\in\Z^d}g(\s^{-1}k)\vp(x-\s^{-1}k)=(g*\vp)(x),\quad x\in \R^d.
\end{equation}
Further, if $\vp$ and parameter $s>0$ are such that
\begin{equation}\label{conv4}
\mathcal{F}^{-1}\(\frac{\eta(1-\widehat{\vp})}{|\cdot|^s}\)\in L_1(\R^d)
\end{equation}
for some function $\eta\in C_0^\infty(\R^d)$ equal to 1 on $(-\tfrac 12,\tfrac 12)^d$, then
\begin{equation}\label{conv3}
 \|g-S_\s^\vp g\|_p\le C\|(-\Delta)^{s/2}g\|_p,
\end{equation}
where $\Delta$ is the Laplacian and the constant $C$ is independent of $\s$ and $g$.
\end{proposition}

\begin{proof}
Equality~\eqref{conv1} follows directly form the Poisson summation formula.
Further, by~\eqref{conv1}, we have
\begin{equation}\label{conv2}
 \|g-S_\s^\vp g\|_p=\|g-g*\vp_\s\|_p,
\end{equation}
where $\vp_\s(x)=\s^{d}\vp(\s x)$. At the same time, condition~\eqref{conv4} and Young's convolution inequality yield
\begin{equation}\label{conv5}
 \|g-g*\vp_\s\|_p\le C\|(-\Delta)^{s/2}g\|_p,
\end{equation}
where $\Delta$ is the Laplacian and $C$ is a constant independent of $\s$ and $g$. Finally, \eqref{conv2} and~\eqref{conv5} give~\eqref{conv3}.
\end{proof}



\subsection{Sampling operators with band-limited kernels} The classical Whittaker--Kotelnikov--Shannon sampling expansion is defined by
$$
S_\s^{\rm sinc} f(x)=\sum_{k\in\Z^d}f\(\s^{-1}k\){\rm sinc}(\s x-k),\quad \s>0,
$$
where
$$
{\rm sinc}(x)=\prod_{j=1}^d \frac{\sin \pi x_j}{\pi x_j}.
$$
One of the main tools for studying the properties of these operators is the Plancherel-Polya inequality (see, e.g.,~\cite[1.5.5]{ST86}). This inequality states that there exist positive constants $C_1$ and $C_2$ such that, for all $g\in \mathcal{B}_p^\s(\R^d)$, $1<p<\infty$, and $\s>0$,
\begin{equation}\label{mzS}
  C_1\s^{-d}\sum_{k\in \Z^d}|g(\s^{-1}k)|^p\le \|g\|_{p}^p\le C_2\s^{-d}\sum_{k\in \Z^d}|g(\s^{-1}k)|^p.
\end{equation}
Inequalities~\eqref{mzS} and the fact that
\begin{equation*}
  S_\s^{\rm sinc} f\(\s^{-1}k\)=f\(\s^{-1}k\),\quad k\in\Z^d,
\end{equation*}
 imply that $S_\s^{\rm sinc}$ satisfies assumptions~\eqref{c1} and~\eqref{c1'} for $X_\s=(\s^{-1}k)_{k\in \Z^d}$ as well as assumption~\eqref{c2'} for each $s\in \N$. Indeed, the fulfilment of~\eqref{c1} and~\eqref{c1'} is straightforward, while~\eqref{c2'} is ensured by the fact that $S_\s^{\rm sinc}(g)=g$ for each $g\in \mathcal{B}_p^{\s/2}(\R^d)$, see Proposition~\ref{propA3}. 
Thus, taking into account Lemma~\ref{leJB+}, we obtain that Theorems~\ref{th1+}--\ref{th1RG}, \ref{th4}, and~\ref{th5} are valid for $S_\s^{\rm sinc}$. In particular, Corollary~\ref{cor3} can be written as follows.

\begin{corollary}\label{cor3S}
  Let $f\in L_p(\R^d)\cap \ell_p(X_\s)$, $1<p<\infty$, $r,s\in \N$, $s\le 2r$, and $\a\in (0,s)$. Then the following properties are equivalent:
  \begin{enumerate}
    \item[$(i)$]  $\|f-S_\s^{\rm sinc} f\|_{p}=\mathcal{O}(\s^{-\a})$,
    \item[$(ii)$]  $\|f_{1/2\s,\,r}-f\|_{\ell_p(X_\s)}+\w_s(f,\s^{-1})_{p}=\mathcal{O}(\s^{-\a})$,
    \item[$(iii)$] $|S_\s^{\rm sinc} f|_{W_p^s}=\mathcal{O}(\s^{s-\a})$.
  \end{enumerate}
\end{corollary}

As a next standard example, we consider the sampling operator generated by the Riesz kernel
\begin{equation*}
\vp(x)=\rho_{s,\d}(x)=\mathcal{F}^{-1}\((1-|\tfrac 43\xi|^s)_+^\d\)(x).
\end{equation*}
Using~\cite{RS10} (see also~\cite{K12}), we get that $\rho_{s,\d}$ with $\d>\frac{d-1}{2}$ and $s>0$ satisfies~\eqref{conv4}. Therefore, by Proposition~\ref{propA3}, we have that assumption~\eqref{c2'Y} with $|g|_{Y_p^s}=\|(-\Delta)^{s/2}g\|_p$ is fulfilled for  $S_\s^{\rho_{s,\d}}$. It is also not difficult to see that~\eqref{jia1} and~\eqref{jia2} hold for $\rho_{s,\d}$, which by Proposition~\ref{propA1A2} implies that assumptions~\eqref{c1} and~\eqref{c1'} are also fulfilled. Now, denoting
$$
K_s^\Delta(f,t)_p:=\inf_{g}\(\|f-g\|_p+t^s \|(-\Delta)^{s/2}g\|_p\),
$$
and applying Theorems~\ref{th1+Y} and~\ref{th3Y} as well as Lemma~\ref{leJB+}, we obtain the following result.


\begin{corollary}\label{cor3SR}
  Let $f\in L_p(\R^d)\cap \ell_p(X_\s)$, $1<p<\infty$, $r\in \N$, $0<s\le 2r$, $\d>\tfrac{d-1}2$, and $\a\in (0,s)$. Then the following properties are equivalent:
  \begin{enumerate}
    \item[$(i)$]  $\|f-S_\s^{\rho_{s,\d}} f\|_{p}=\mathcal{O}(\s^{-\a})$,
    \item[$(ii)$]  $\|f_{1/2\s,\,r}-f\|_{\ell_p(X_\s)}+K_s^\Delta(f,\s^{-1})_p=\mathcal{O}(\s^{-\a})$.
  \end{enumerate}
\end{corollary}

\subsection{Sampling operators with time-limited kernels} Another typical example of sampling operators of the form~\eqref{SO} that satisfy assumptions~\eqref{c1}--\eqref{c2'} is given by operators generated by $B$-splines. Recall that the univariate $B$-spline of order $r\in \N$, $r\ge 2$, is defined by
$$
B_r(u)=\frac1{(r-1)!}\sum_{j=0}^r (-1)^j\binom{r}{j}\(\frac r2+u-j\)_+^{r-1},\quad u\in \R,
$$
where $(\cdot)_+$ denotes the positive part. In particular, one has $B_2(u)=(1-|u|)_+$.
We set
$$
S_{\s}^{\b_r} f(x)=\sum_{k\in \Z^d}f(\s^{-1}k)\b_r(\s x-k),\quad \s>0,
$$
where
$$
\b_r(x)=\prod_{j=1}^b B_r(x_i).
$$
Noting that $\widehat{\b_r}(\xi)=\prod_{j=1}^d\(\frac{\sin \pi\xi_j}{\pi\xi_j}\)^r$ and  ${\rm supp}\,\b_r\subset [-\frac{r}{2},\frac{r}{2}]^d$,
we obtain that relations~\eqref{jia1} and~\eqref{jia2} hold for $\b_r$. Therefore, by Proposition~\ref{propA1A2}, the operator $S_{\s}^{\b_r} $ satisfies assumptions~\eqref{c1} and~\eqref{c1'} with $X_\s=(\s^{-1}k)_{k\in \Z^d}$.
At the same time, by~\cite{BBSV06} (see also~\cite{B88}), we have the following error estimate
\begin{equation*}
  \|f-S_{\s}^{\b_r}  f\|_{p}\le C\tau_1(f,\s^{-1})_{p},\quad \s>0,
\end{equation*}
which together with property~$(e')$ and the Bernstein inequality~\eqref{Ber} gives
\begin{equation*}
  \|g-S_{\s}^{\b_r} g\|_{p}\le C\s^{-1}|g|_{W_p^1},\quad g\in \mathcal{B}_p^\s(\R^d),
\end{equation*}
where the constant $C$ is independent of $\s$ and $g$.
Hence, assumption~\eqref{c2'} also holds for $S_{\s}^{\b_r} $. Now, noting that 
 $\Phi_{\s}^{\b_r}={\rm span}\,\(\b_{r}(\s x-k)\)_{k\in \Z^d}$ with $\Phi_{0}^{\b_r}=\{0\}$ satisfies inequalities~\eqref{js}, \eqref{bs}, and~\eqref{bebe2}, see, e.g.,~\cite[Ch.~5]{DL}, we obtain the following special case of Corollary~\ref{cor3}.

\begin{corollary}\label{cor3Sr}
  Let $f\in L_p(\R^d)\cap \ell_p(X_\s)$, $1\le p<\infty$, $r\in \N$, and $\a\in (0,1)$. Then the following properties are equivalent:
  \begin{enumerate}
    \item[$(i)$]  $\|f-S_{\s}^{\b_r} f\|_{p}=\mathcal{O}(\s^{-\a})$,
    \item[$(ii)$]  $\|f_{1/2\s,\,r}-f\|_{\ell_p(X_\s)}+\w_1(f,\s^{-1})_{p}=\mathcal{O}(\s^{-\a})$.
  \end{enumerate}
\end{corollary}

\subsection{Sampling operators generated by Gaussian}
Consider the operator
$$
S_\s^\psi f(x)=\sum_{k\in \Z^d} f(\s^{-1}k)\psi(\s x-k),\quad \s>0,
$$
where
$\psi(x)=e^{-\pi |x|^2}$. Applying~\cite[Theorem~16]{KS20}, we can verify that for each $g\in \mathcal{B}_p^{\l \s}(\R^d)$ with $1\le p<\infty$ and $\l \in (0,1/2)$,
\begin{equation}\label{ga1}
  \|g-S_\s^\psi g\|_p\le C\s^{-1}|g|_{W_p^1},
\end{equation}
 where the constant $C$ is independent of $\s$ and $g$. Therefore, $S_\s^\psi$ satisfies~\eqref{c2'} with $s=1$. Furthermore, by Proposition~\ref{propA1A2}, it also satisfies assumptions~\eqref{c1} and~\eqref{c1'}.

Now we denote $\Phi_{0,p}=\{0\}$ and, for $\s>0$,
 $$
 \Phi_{\s,p}=\bigg\{g\,:\,g(x)=\sum_{k\in \Z^d}c_k\psi(\s x-k),\quad (c_k)_{k\in\Z^d}\in \ell_p\bigg\}.
 $$
Inequalities~\eqref{ga1}, \eqref{leJB.1+}, and~\eqref{NS} imply that $E(f, \Phi_{\s,p})_p\le C\s^{-1}|f|_{W_p^1}$. We now show that, for all multi-indices $\a\in\Z_+^d$, the following Bernstein-type inequality holds:
\begin{equation}\label{berG}
  \| D^\a g\|_p\le C\s^{|\a|_1}\|g\|_p,\quad g \in  \Phi_{\s,p}^\psi,
\end{equation}
where the constant $C$ is independent of $\s$ and $g$. Indeed, it is clear that $D^\a \psi$ satisfies~\eqref{jia1} and $\psi$ satisfies~\eqref{jia2}. Thus, applying Proposition~\ref{propA1A2}, we get for all $g\in \Phi_{1,p}^\psi$
\begin{equation*}
  \| D^\a g\|_p^p\le C\sum_{k\in\Z^d}|c_k|^p \le C\|g\|_p^p.
\end{equation*}
It remains only to replace $g(x)$ with $g(\s x)$ to get~\eqref{berG}.

Now, applying Theorems~\ref{th1+} and~\ref{th3}, we obtain the following corollary.
\begin{corollary}\label{corGa}
  Let $f\in L_p(\R^d)\cap \ell_p(X_\s)$, $1\le p<\infty$, and $\a\in (0,1)$. Then the following properties are equivalent:
  \begin{enumerate}
    \item[$(i)$]  $\|f-S_{\s}^{\psi} f\|_{p}=\mathcal{O}(\s^{-\a})$,
    \item[$(ii)$]  $\|f_{1/2\s,\,r}-f\|_{\ell_p(X_\s)}+\w_1(f,\s^{-1})_{p}=\mathcal{O}(\s^{-\a})$.
  \end{enumerate}
\end{corollary}

Similar results can also be obtained for irregular sampling points. In what follows,  we restrict ourselves to the case of univariate interpolation operators in $L_2(\R)$. Let $X=(x_j)_{j\in \Z}\subset \R$, $x_j<x_{j+1}$,  be a Riesz basis sequence for $L_2[-\tfrac12,\tfrac12]$, i.e., there exists a constant $B>0$ such that
$$
\frac1B\bigg(\sum_{j\in \Z}|c_j|^2\bigg)^{1/2}\le \bigg\|\sum_{j\in \Z} c_j e^{2\pi ix_j(\cdot)}\bigg\|_{L_2[-\tfrac12,\tfrac12]}   \le B\bigg(\sum_{j\in \Z}|c_j|^2\bigg)^{1/2}
$$
for every sequence $(c_k)\in \ell_2$.  Recall that a necessary condition for $X$ to form a Riesz basis is that it be a quasi-uniform, meaning that there exist $0< q\le Q<\infty$ such that $q\le |x_{j+1}-x_j|\le Q$. A classical sufficient condition is a famous Kadec's 1/4-theorem (see~\cite{Ka64}), which states that if $X$ is such that $\sup_{j\in\Z}|x_j-j|<1/4$, then  $X$ is a Riesz basis sequence for $L_2[-\tfrac12,\tfrac12]$.

Let
$$
I^Xf(x)=\sum_{j\in \Z} a_j \psi(x-x_j), \quad a_j=a_j(f)\in \C,
$$
be an interpolation operator such that $I^Xf(x_k)=f(x_k)$ for each $k\in \Z$. It was proved in~\cite{Ha15} that for any $r\in \N$ and $f\in W_2^r(\R)$
\begin{equation}\label{Ha1}
  \|f-I^Xf\|_2\le C|f|_{W_2^r},
\end{equation}
where the constant $C$ does not depend on $f$.

We set
$$
I^X_\s f(x)=I^X f^{1/\s}(\s x),
$$
where $f^{1/\s}(x)=f(x/\s)$. Then, $I^X_\s f(\s^{-1}x_j)=f(\s^{-1}x_j)$, $j\in \Z$, and, by the standard substitutions, we obtain from~\eqref{Ha1} that
\begin{equation*}
  \|f-I_\s^Xf\|_2\le C\s^{-r}|f|_{W_2^r},\quad f\in W_2^r,
\end{equation*}
where the constant $C$ does not depend on $\s$ and $f$. This implies that  $I_\s^X$ satisfies assumption~\eqref{c2'}.

Now we consider the set
$$
\Phi_\s(X)=\bigg\{g\,:\,g(x)=\sum_{j\in \Z} c_j \psi(\s x-x_j),\quad (c_j)\in \ell_2\bigg\}.
$$
In what follows, we denote $X_\s=(\s^{-1}x_j)_{j\in \Z}$ and $X=X_1$. By Proposition~2 and Theorem~3 in~\cite{HL18}, we have that for all $g=\sum_{j\in \Z} c_j \psi(\cdot-x_j)\in \Phi_1(X)$ the following equivalence hold
\begin{equation}\label{Ha3}
  \|g\|_2\asymp \|g\|_{\ell_2(X)}\asymp \|(c_j)\|_{\ell_2},
\end{equation}
where $\asymp$ denotes two-sided inequalities with constants independent of $g$. The first equivalence in~\eqref{Ha3} together with the standard substitution gives
\begin{equation*}
  \| I_\s^X g\|_2\asymp \|f\|_{\ell_2(X_\s)}.
\end{equation*}
Thus, $I_\s^X $ satisfies assumptions~\eqref{c1} and~\eqref{c1'}.

For our purposes, we also need to establish a Bernstein-type inequality for elements of $\Phi_\s(X)$. Let $g$ be the same as in~\eqref{Ha3}. By~\cite[Lemma~8.1]{UZ24}, we have
\begin{equation}\label{HaB1}
  \|g^{(r)}\|_2\le C\|(c_j)\|_{\ell_2},\quad r\in \N,
\end{equation}
where the constant $C$ does not depend on $g$. At the same time by~\eqref{Ha3} we have
\begin{equation}\label{HaB2}
  \|(c_j)\|_{\ell_2}\le C\|g\|_2.
\end{equation}
Thus, replacing $g$ by $g(\s x)$, we derive from~\eqref{HaB1} and~\eqref{HaB2} that
\begin{equation*}
  \|g^{(r)}\|_2\le C\s^{r}\|g\|_2,\quad g \in  \Phi_{\s},
\end{equation*}
where the constant $C$ is independent of $\s$ and $g$.

Finally, applying Theorems~\ref{th1+} and~\ref{th3}, we get the following corollary.
\begin{corollary}\label{corHa}
  Let $f\in L_2(\R)\cap \ell_2(X_\s)$, $r\in \N$,  and $\a\in (0,r)$. Then the following properties are equivalent:
  \begin{enumerate}
    \item[$(i)$]  $\|f-I_{\s}^X f\|_{p}=\mathcal{O}(\s^{-\a})$,
    \item[$(ii)$]  $\|f_{1/2\s,\,r}-f\|_{\ell_2(X_\s)}+\w_r(f,\s^{-1})_{2}=\mathcal{O}(\s^{-\a})$.
  \end{enumerate}
\end{corollary}

\end{document}